\documentclass[11pt]{article}
\usepackage{amsfonts,amsmath,amsthm,amssymb,amscd}
\usepackage[all]{xy}
\usepackage{color}

\newtheorem{theorem}{Theorem}
\newtheorem{lemma}[theorem]{Lemma}

\newcommand{\C}{\mathbf C}

\newcommand{\Q}{\mathbf Q}
\newcommand{\Z}{\mathbf Z}
\newcommand{\F}{\mathbf F}
\renewcommand{\Re}{{\operatorname{Re}}}

\newcommand{\trace}{{\operatorname{trace}}}
\newcommand{\Sym}{{\operatorname{Sym}}}

\newcommand{\Li}{{\operatorname{Li}}}
\begin{document}
\title{Zeros of Witten zeta functions and absolute limit}
\author{N. Kurokawa and H. Ochiai}
\date{}
\maketitle

\section{Introduction}
The Witten zeta function
\[
\zeta_G^W(s) = \sum_{\rho \in \hat{G}} \deg(\rho)^{-s}
\]
was introduced by Witten \cite{W} in 1991, 
where $G$ is a compact topological group
and $\hat{G}$ denotes the unitary dual,
that is, the set of equivalence classes of irreducible unitary representations.
The example
\[
\zeta^W_{SU(2)}(s)
= \sum_{m=0}^\infty \deg(\Sym^m)^{-s}
= \sum_{n=1}^\infty n^{-s} = \zeta(s),
\]
where $\zeta(s)$ denotes the Riemann zeta function,
suggests fine properties for general case.
In fact, Witten showed arithmetical interpretation for
$\zeta^W_{SU(n)}(2m)$ ($m=1,2,3,\dots$) containing Euler's result (\cite{E1} 1735)
\[
\zeta^W_{SU(2)}(2m) \in \pi^{2m} \Q.
\]
In this paper we look at the opposite side:
special values at negative integers such as
\begin{align}
\zeta^W_{SU(2)}(-1) &= \mbox{``\ $\displaystyle\sum_{n=1}^\infty n$\ ''} = -\frac1{12}, \\
\zeta^W_{SU(2)}(-2) &= \mbox{``\ $\displaystyle\sum_{n=1}^\infty n^2$\ ''} = 0
\end{align}
due to Euler \cite{E2}(1749).
We notice that the value
\[
\mbox{``\ $\displaystyle\sum_{n=1}^\infty n$\ ''}=-\frac1{12}
\]
appears as the one-dimensional Casimir energy:
see Casimir \cite{C} and Hawking \cite{H}.
The equality
\[
\mbox{``\ $\displaystyle\sum_{n=1}^\infty n^2$\ ''} = 0
\]
means the vanishing of the two-dimensional Casimir energy.

We notice that
\[
\zeta^W_G(-2) = \left| G \right|
\]
when $G$ is a finite group.
We conjecture that
\begin{equation}
\zeta^W_G(-2)=0
\end{equation}
for infinite groups $G$.

For deeper understanding of the situation, 
we introduce a new zeta function
(Witten $L$-function)
\begin{equation}
\zeta^W_G(s,g) = \sum_{\rho \in \widehat{G}} \frac{\trace(\rho(g))}{\deg(\rho)} \deg(\rho)^{-s}
\end{equation}
where $G$ is a compact topological group,
$g$ is an element of $G$,
$\widehat{G}$ is the set of equivalence classes of
irreducible ($\C$-valued) representations of $G$,
$\deg(\rho)$ is the degree (the dimension) of 
an irreducible representation $\rho \in \widehat{G}$.
Note that $\trace(\rho(g))$ is the character of the representation $\rho$.
This Witten zeta function $\zeta_G^W(s,g)$
reduces to 
the (usual) Witten zeta function
when we specialize $g$ to be the identity element $1 \in G$:
\[
\zeta^W_G(s) = \zeta^W_G(s,1).
\]
In the case of a finite group $G$ we have
\[
\zeta^W_G(-2,g) = 
\left\{ \begin{array}{ll}
\left| G \right| & \mbox{ if } g=1, \\
0 & \mbox{ otherwise}.
\end{array}
\right.
\]
We conjecture that
\begin{equation}
\zeta^W_G(-2,g)=0
\end{equation}
when $G$ is an infinite group.
The following result treats the case $G=SU(2)$.

\begin{theorem}\label{Theorem:SU(2)}
{\color{black} Suppose $g \in SU(2)$ is conjugate to
$
\left( \begin{array}{cc}
e^{i \theta} & 0 \\ 0 & e^{-i\theta} \end{array} \right)$
with $0 \le \theta \le \pi$.}
\begin{itemize}
\item[{\rm(1)}]
{\color{black}We have an expression}
\[
\zeta^W_{SU(2)}(s,g) 
= \sum_{n=1}^\infty \frac{\sin(n \theta)}{n \sin \theta} n^{-s}
\]
in $\Re(s)>1$ 
{\color{black}The function $\zeta^W_{SU(2)}(s,g)$ in $s$ 
has a meromorphic continuation to the whole complex plane.}
\item[{\rm(2)}]
For a positive even integer $m$,
we have
$\zeta_{SU(2)}^W(-m,g)=0$ for all $g \in SU(2)$.
Moreover, $s=-2$ is a simple zero of $\zeta_{SU(2)}^W(s,g)$,
and the first derivative at $s=-2$ is given as
\[
\frac{\partial \zeta^W_{SU(2)}}{\partial s}
(-2,g)
=
\left\{\begin{array}{ll}
-\frac{\zeta(3)}{4\pi^2} & \mbox{ if } \theta=0, \\
\frac{1}{4\pi \sin \theta}
\left( \zeta(2,\frac{\theta}{2\pi}) - \frac{\pi^2}{2 \sin^2 \frac{\theta}{2}} \right) >0
& \mbox{ if } 0 < \theta < \pi, \\
\frac{7\zeta(3)}{4\pi^2} & \mbox{ if } \theta=\pi.
\end{array}\right.
\]
Here $\zeta(s,x)$ denotes the Hurwitz zeta function.

\item[{\rm(3)}]
{\color{black} The special value at $s=-1$ is given as}
\[
\zeta^W_{SU(2)}(-1, g)
=
\left\{\begin{array}{ll}
-\frac1{12} & \mbox{ if } \theta=0, \\
\frac{1}{4 \sin^2\frac{\theta}{2}}
& \mbox{ if } 0 < \theta < \pi, \\
\frac14
& \mbox{ if } \theta=\pi.
\end{array}\right.
\]
\end{itemize}
\end{theorem}
{\color{black} We now introduce a `multi'-version
of Witten $L$-function.
For $g_1,\ldots,g_r \in G$, 
we define}
\begin{align*}
\zeta^W_G(s;g_1,\dots,g_r) 
& := \sum_{\rho \in \widehat{G}} 
\frac{\trace(\rho(g_1))}{\deg(\rho)} \cdots
\frac{\trace(\rho(g_r))}{\deg(\rho)} \times \deg(\rho)^{-s} \\
& = \sum_{\rho \in \widehat{G}}
\frac{\trace(\rho(g_1)) \cdots \trace(\rho(g_r))}{\deg(\rho)^{s+r}}.
\end{align*}

{\color{black} It is natural to ask whether the vanishing 
$\zeta^W_G(-2;g_1,\dots,g_r) \overset{?}{=} 0$
of the special value at $s=-2$ for this generalization
holds.}
{\color{black} We have a partial answer to this question.}
\begin{theorem}\label{r=2}
We have
$\zeta^W_{SU(2)}(-m;g_1,g_2) =0$
for $g_1, g_2 \in SU(2)$,
and a positive even integer $m$.
\end{theorem}
We also give an example of the non-vanishing for the case $r=3$:
for some $g \in SU(2)$, we prove that
$\zeta^W_{SU(2)}(-2;g,g,g) \neq 0$.
These results related with the Lie group $SU(2)$
are given in Section~\ref{SU(2)}.


We report further examples of zeros of Witten zeta functions for infinite groups.

\begin{theorem}\label{Theorem:SU(3)}
$\zeta^W_{SU(3)}(s)=0$ for $s=-1,-2,\dots.$
\end{theorem}
The proof of this theorem is given in Section~\ref{SU(3)}.

The next example is not a Lie group, but
a totally disconnected group.
Let $\Z_p$ be the $p$-adic integer ring for a prime number $p$.
\begin{theorem}\label{Theorem:SL(Zp)}
Suppose $p \neq 2$. 
Then
$\zeta^W_{SL_2(\Z_p)}(s)=0$ for $s=-1, -2$.
\end{theorem}
Now we consider the congruence subgroups.
For a positive integer $m$,
we define a subgroup of $SL_3(\Z_p)$ of finite index by
\[
SL_3(\Z_p)[p^m] = \ker(SL_3(\Z_p)  \rightarrow SL_3(\Z_p/(p^m))).
\]
\begin{theorem}
Suppose $p \neq 3$.
\begin{itemize}
\item[{\rm(1)}]
\begin{align*}
\zeta_{SL_3(\Z_p)[p^m]}^W(s)
&= p^{8m} \frac{(1-p^{-2-s})(1-p^{-1-s})}{(1-p^{1-2s})(1-p^{2-3s})}\\
& \times \left(
1+(p^{-1}+p^{-2}) p^{-s} + (1+p^{-1}) p^{-2s}+p^{-2-3s}
\right).
\end{align*}
\item[{\rm(2)}]
$\zeta_{SL_3(\Z_p)[p^m]}^W(s)=0$ for $s=-1, -2$.
\item[{\rm(3)}]
\[
\zeta^W_{SL_3(\Z_1)[1^m]}(s) = \frac{(s+1)(s+2)}{(s-\frac12)(s-\frac23)}.
\]
\end{itemize}
\end{theorem}
Here we interpret that
if $\zeta_{SL_3(\Z_p)[p^m]}^W(s)$ has an expression
as an analytic function on $p$,
and there is a limit $p\to 1$,
then its limit is denoted by
\[
\zeta^W_{SL_3(\Z_1)[1^m]}(s)  = \lim_{p \to 1} \zeta_{SL_3(\Z_p)[p^m]}^W(s).
\]
These results on totally disconnected groups
are given in Section~\ref{Zp}.

\section{$SU(2)$}\label{SU(2)}

\subsection{Parametrization of irreducible representations of $SU(2)$}
The set of equivalence classes, $\widehat{G}$, of irreducible unitary representations of $G=SU(2)$
is parametrized by the set of natural numbers.
For a natural number, we denote by $\rho=\rho_n \in \widehat{G}$,
the corresponding irreducible representation of $G$.
For a $g=
\left( \begin{array}{cc} 
e^{i\theta} & 0 \\
0 & e^{-i\theta}
\end{array} \right) \in G$,
we have the character formula
\begin{equation}
\trace(\rho(g))= e^{i(n-1)\theta}+e^{i(n-3)\theta}+\cdots+e^{i(3-n)\theta}+e^{i(1-n)\theta}
\end{equation}
and the degree
\begin{equation}
\deg(\rho) = \trace(\rho(I_2)) = n,
\end{equation}
where $I_2 = \textstyle \left( \begin{array}{cc} 1& 0 \\ 0 & 1 \end{array} \right) \in SU(2)$
is the identity matrix.
We also see that
$\trace(\rho(-I_2)) = (-1)^{n-1}n$.

We start from $g = \pm I_2 \in SU(2)$.
In these cases,
$\zeta_{SU(2)}^W(s,g)$ is written in terms of Riemann zeta function.
We see that
$\zeta_{SU(2)}^W(s, I_2) =\zeta(s)$,
and
\begin{equation}
\zeta_{SU(2)}^W(s, -I_2) = \sum_{n=1}^\infty \frac{(-1)^{n-1}}{n^s} = (1-2^{1-s}) \zeta(s).
\end{equation}

\subsection{Poly-logarithm function}
We recall the poly-logarithm
\begin{align*}
Z(s,x) 
&= \sum_{n=1}^\infty \frac{x^n}{n^s},
\end{align*}
which is written also as $\Li_s(x)$ in literature.
This series converges if
$\left| x \right|<1$ and $s \in \C$,
or $\left| x \right|=1$ and $\Re(s)>1$.
In the following, we restrict to the case $\left| x \right|=1$.
\begin{theorem}
Suppose $\left| x \right|=1$ and $x \neq 1$.
Then $Z(s,x)$ is analytically continued to a holomorphic function on $s \in \C$.
Moreover, for every non-negative integer $m$,
the function $Z(-m,x)$ can be expressed by a rational function in $x$.
The first several examples are
\[
Z(0,x) = \frac{x}{1-x},
\quad
Z(-1, x) = \frac{x}{(1-x)^2}, 
\quad
Z(-2,x) = \frac{x(1+x)}{(1-x)^3},
\ldots.
\]
\end{theorem}
\begin{proof}
For $\Re(s)>1$, we have
\begin{align*}
Z(s,x) 
&= x + \frac{x^2}{2^s} + \sum_{n=3}^\infty \frac{x^n}{n^s} \\
&= x + \frac{x^2}{2^s} + \sum_{n=2}^\infty \frac{x^{n+1}}{(n+1)^s} \\
&= x + \frac{x^2}{2^s} + \sum_{n=2}^\infty x^{n+1} n^{-s}(1+n^{-1})^{-s} \\
&= x + \frac{x^2}{2^s} + \sum_{n=2}^\infty x^{n+1} n^{-s} \sum_{k=0}^\infty \binom{-s}{k} n^{-k}
\\
&= x + \frac{x^2}{2^s} + x \sum_{k=0}^\infty  \binom{-s}{k} ( Z(s+k,x) -x) \\
&= x + \frac{x^2}{2^s} + x(Z(s,x)-x) + x \sum_{k=1}^\infty  \binom{-s}{k} ( Z(s+k,x) -x).
\end{align*}
This shows
\begin{align}
(1-x) Z(s,x) 
&= x+x^2(2^{-s}-1) + x \sum_{k=1}^\infty  \binom{-s}{k} ( Z(s+k,x) -x).
\end{align}
By the estimates of binomial coefficients,
the right-hand side converges absolutely on the right-half plane $\Re(s) > 0$.
This shows the analytic continuation of $Z(s,x)$ to $\Re(s)>0$.
Repeating this argument, we obtain the analytic continuation to whole $s \in\C$.
To substitute $s=-m$ with $m=0,1,\dots$,
we have the recursion equation
\begin{align}
(1-x) Z(-m,x) = x+x^2(2^m-1) + x \sum_{k=1}^m  \binom{m}{k} ( Z(-(m-k),x) -x).
\end{align}
\end{proof}
First several examples show
\begin{align*}
Z(-3,x) & = \frac{x(1+4x+x^2)}{(1-x)^4},
\qquad
Z(-4,x) = \frac{x(1+x)(1+10x+x^2)}{(1-x)^5},\\
Z(-5,x) &= \frac{x(1+26x+66x^2+26x^3+x^4)}{(1-x)^6}.
\end{align*}
These examples seem to show 
\begin{lemma}
Suppose $\left| x \right| =1$ with $x \neq 1$.
Then
\begin{align}\label{eq:Z(0)}
Z(0,x) + Z(0,x^{-1}) = -1,
\end{align}
and for every positive integer $m$, 
\begin{equation}\label{eq:Z-funct}
Z(-m,x) + (-1)^m Z(-m, x^{-1}) = 0.
\end{equation}
\end{lemma}
\begin{proof}
We start from [Jonqui\`ere 1880]
\begin{equation}\label{eq:Milnor}
e^{-\pi i s/2} Z(s,e^{i\theta}) + e^{\pi i s/2} Z(s,e^{-i\theta})
= \frac{(2\pi)^{s}}{\Gamma(s)} \zeta(1-s, \frac{\theta}{2\pi})
\end{equation}
in Milnor \cite{M}.
Putting $s=-m$ with $m=1,2,\dots$,
we have
\[
e^{\pi i m/2} Z(-m, e^{i\theta}) + e^{-\pi i m/2} Z(-m, e^{-i\theta}) =0.
\]
\end{proof}
We remark that
$Z(0,1) = \zeta(0)=-1/2$.
In this sense, the formula (\ref{eq:Z(0)}) is valid also for $x=1$.


\subsection{An example}

\begin{equation}
Z(-1,e^{i\theta}) 
= \frac{1}{(e^{-i\theta/2}- e^{i \theta/2})^2} 
= -\frac{1}{4 \sin^2(\theta/2)}.
\end{equation}
and this shows 
\begin{equation}
\label{eq:polylog-even}
\Li_{-1}(e^{-i\theta})=\Li_{-1}(e^{i\theta}),
\end{equation}
an even function in $\theta$.

\subsection{Proof of Theorem~\ref{Theorem:SU(2)}(1) and analytic continuation}

Now we consider regular elements in $SU(2)$.
Suppose $0 < \theta < \pi$.
Then
we have,
for $\Re(s)>1$,
\begin{align*}
\zeta_{SU(2)}^W\left( s, 
\left( \begin{array}{cc} 
e^{i\theta} & 0 \\
0 & e^{-i\theta}
\end{array} \right)
 \right)
&=
\sum_{n=1}^\infty \frac{e^{i n\theta}-e^{-i n\theta}}{e^{i\theta}-e^{-i \theta}}\frac{1}{n} n^{-s} \\
&=\frac{1}{e^{i\theta}-e^{-i \theta}} \sum_{n=1}^\infty 
\left( \frac{e^{in\theta}}{n^{s+1}} - \frac{e^{-in\theta}}{n^{s+1}} \right) 
\\
&=
\frac{1}{e^{i\theta}-e^{-i \theta}}\left\{ Z(s+1,e^{i\theta}) - Z(s+1,e^{-i\theta}) \right\} \\
&=
\frac{1}{2i \sin \theta} \left\{ Z(s+1,e^{i\theta}) - Z(s+1,e^{-i\theta}) \right\},
\end{align*}
and the right-hand side has meromorphic continuation to whole $s \in \C$.

Note that we interpret 
\begin{equation}
\frac{\sin(n\theta)}{n \sin \theta} = \left\{
\begin{array}{ll}
1 & \mbox{ if } \theta = 0, \\
(-1)^{n-1} & \mbox{ if } \theta= \pi.
\end{array}
\right.
\end{equation}


\subsection{Proof of Theorem~\ref{Theorem:SU(2)}(2); vanishing}

For $g = \pm I_2$
and for positive even integer $m$,
we obtain
$\zeta_{SU(2)}^W(-m,\pm I_2)=0$ from $\zeta(-m)=0$.

For $g \neq \pm I_2$,
suppose $0<\theta<\pi$.
Then for a positive integer $m$, we have
\begin{equation}
\zeta_{SU(2)}^W(-m, g) 
=
\frac{1}{2i \sin \theta} \left( Z(1-m,e^{i\theta}) - Z(1-m,e^{-i\theta}) \right).
\end{equation}
This is zero for even $m$ by the formula (\ref{eq:Z-funct}).

\subsection{Proof of Theorem~\ref{Theorem:SU(2)}(2), first derivative}

We see that 
\begin{equation}
\frac{1}{\Gamma(s)} =\frac{s(s+1)}{\Gamma(s+2)}
\end{equation}
shows that
\begin{equation}
\frac{1}{\Gamma(s)} =-(s+1) + O((s+1)^2), \quad (s\to -1).
\end{equation}

We again start from the formula (\ref{eq:Milnor})
\[
e^{-\pi i s/2} Z(s,x) + e^{\pi i s/2} Z(s,x^{-1})
= \frac{(2\pi)^{s}}{\Gamma(s)} \zeta(1-s, \frac{\theta}{2\pi})
\]
with $x=e^{i\theta}$.
Taking $\left.\frac{\partial}{\partial s}\right|_{s=-1}$ in this formula,
we have
\begin{align*}
& i \frac{\partial Z}{\partial s} (-1,x) 
+ (-i) \frac{\partial Z}{\partial s} (-1,x^{-1}) \\
& \quad + (-\pi i/2)(i) Z(-1,x) + (\pi i/2)(-i) Z(-1,x^{-1}) 
= (2\pi)^{-1} (-1) \zeta(2,\frac{\theta}{2\pi}).
\end{align*}
Then
\begin{align}
i\times 2 i \sin \theta \times \frac{\partial \zeta_{SU(2)}^W}{\partial s} (-2,g) 
= -\pi Z(-1,e^{i\theta})
- \frac{1}{2\pi} \zeta(2,\frac{\theta}{2\pi}),
\end{align}
and
\begin{align}
4\pi \sin \theta \times \frac{\partial \zeta_{SU(2)}^W}{\partial s} (-2,g) 
= 2\pi^2 L(-1,e^{i\theta})
+ \zeta(2,\frac{\theta}{2\pi})
\end{align}

We have
\begin{equation}
\zeta(2,t) + \zeta(2,1-t)= \frac{\pi^2}{\sin^2(\pi t)}
\end{equation}
since the left-hand side is equal to
\begin{equation}
\sum_{n=0}^\infty \frac{1}{(n+t)^2}
+ \sum_{n=0}^\infty \frac1{(n+1-t)^2} 
=
\sum_{n=-\infty}^\infty \frac{1}{(n+t)^2} 
\end{equation}
which is equal to the right-hand side.
This shows

\begin{align}
8\pi \sin \theta \times \frac{\partial \zeta_{SU(2)}^W}{\partial s} (-2,g)
= 
\zeta(2,\frac{\theta}{2\pi})-\zeta(2,1-\frac{\theta}{2\pi}) >0
\end{align}
since $\frac{\theta}{2\pi} < 1 -\frac{\theta}{2\pi}$.

\subsection{Proof of Theorem~\ref{Theorem:SU(2)}(3)}

\begin{equation}
\zeta_{SU(2)}^W(-1, I_2) = \zeta(-1) =  -\frac1{12}
\end{equation}
and 
\begin{align}
\zeta_{SU(2)}^W\left(-1, 
\left( \begin{array}{cc} 
e^{i\theta} & 0 \\
0 & e^{-i\theta}
\end{array} \right) \right) 
& = \frac{Z(0, x) - Z(0, x^{-1})}{x-x^{-1}} \nonumber \\
& = \frac{-x}{(1-x)^2}
= \frac{1}{4 \sin^2(\theta/2)},
\end{align}
where $x=e^{i\theta}$ for all $0<\theta\le \pi$.

\subsection{An average over the group}

Let $G$ be a finite group.
The normalized Haar measure $dg$ on $G$
is, by definition,
\begin{equation}
\int_G f(g) dg = \frac{1}{\left| G \right|} \sum_{g \in G} f(g).
\end{equation}
Then we see that, for all $s \in \C$,
\begin{align}
\int_G \zeta_G^W(s,g) dg =1,
\end{align}
since
the left-hand side is equal to
\begin{align}
&= \sum_{\rho \in \widehat{G}} \left(\int_G \trace(\rho(g)) dg\right) \deg(\rho)^{-s-1},
\end{align}
where the average is non-zero only for the trivial representation $\rho$.

Now we consider
the case where $G$ is
a compact group
which  is not necessarily a finite group.
Again let $dg$ be the normalized Haar measure of $G$
so that $\int_G dg =1$.
We ask the value 
\begin{equation}
\int_G \zeta_G^W(s,g) dg.
\end{equation}
We can give some example;
\begin{align}
\int_{SU(2)} \zeta_{SU(2)}^W(-2,g) dg &=0, \\
\int_{SU(2)} \zeta_{SU(2)}^W(-1,g) dg &=1.
\end{align}
The latter formula is proved by the Weyl integral formula;
\begin{align}
\int_{SU(2)} \zeta_{SU(2)}^W(-1,g) dg
= \int_0^{\pi} \zeta_{SU(2)}^W(-1,\left( \begin{array}{cc} 
e^{i\theta} & 0 \\
0 & e^{-i\theta}
\end{array} \right)
) \frac{2}{\pi} \sin^2\theta \ d\theta
=1.
\end{align}


\subsection{$r=2$}
We now discuss the properties of a generalization
of Witten zeta functions with several characters.
We give a proof of Theorem~\ref{r=2}.

\begin{proof}
\[
\trace(\rho(g_1))=\frac{x^n-x^{-n}}{x-x^{-1}},
\trace(\rho(g_2))=\frac{y^n-y^{-n}}{y-y^{-1}}
\]
with $x=e^{i \theta_1}$, $y=e^{i\theta_2}$.
In the cases $g_2=\pm I_2$,
we have
\begin{align}
\zeta^W_{SU(2)}(s,g_1, I_2) &= \zeta^W_{SU(2)}(s,g_1), \\
\zeta^W_{SU(2)}(s,g_1, -I_2) &= \zeta^W_{SU(2)}(s,-g_1), 
\end{align}
Then the problem on the special values is reduced to the case
treated in Theorem~\ref{Theorem:SU(2)}(2).

Now we may suppose $x,y \neq \pm 1$.
Then
\begin{align}
& \zeta^W_{SU(2)}(s,g_1,g_2) \nonumber\\
&= \frac{1}{(x-x^{-1})(y-y^{-1})}
\sum_{n=1}^\infty
\frac{(xy)^n +(x^{-1}y^{-1})^n - (x y^{-1})^n - (x^{-1} y)^n}{n^{s+2}} \\
&= \frac{Z(s+2,xy) + Z(s+2,x^{-1}y^{-1}) - Z(s+2,xy^{-1}) - Z(s+2,x^{-1}y)}{(x-x^{-1})(y-y^{-1})}.
\nonumber
\end{align}
This shows
\begin{align}
\zeta^W_{SU(2)}(-2,g_1,g_2)
&=\frac{(Z(0,xy) + Z(0,x^{-1}y^{-1})) - (Z(0,xy^{-1}) + Z(0,x^{-1}y))}{(x-x^{-1})(y-y^{-1})}
\nonumber \\
&=0,
\end{align}
where we have used the formula (\ref{eq:Z(0)}).
\end{proof}
\subsection{$r=3$}\label{r=3}
By the similar computation, we obtain
\begin{align}
& \zeta^W_{SU(2)}(s; g,g,g)\nonumber \\
&= \frac{Z(s+3,x^3) - 3 Z(s+3, x) + 3 Z(s+3,x^{-1}) - Z(s+3,x^{-3})}{(x-x^{-1})^3}.
\end{align}
If $x=i$, then
\[
\zeta^W_{SU(2)}(-2; g,g,g)
= \frac{4Z(1,-i) - 4Z(1,i)}{(2i)^3}
= \frac{-2 \pi i}{-8 i} = \frac{\pi}{4} \neq 0.
\]

\section{$SU(3)$}\label{SU(3)}

\subsection{On analytic continuation}
Let $G$ be a compact semisimple Lie group.
Then the Witten zeta $\zeta_G^W(s)$
has a meromorphic continuation to $\C$.
This is a special case of 
\begin{equation}
\sum_{m_1,\dots,m_r \ge 1} Q(m_1,\dots,m_r) P(m_1,\dots, m_r)^{-s}.
\end{equation}
Analytic continuation of these zeta functions
is discussed in [Mellin 1900], [Mahler 1928].

\subsection{A special value at a negative integer}

Let $n$ be a positive integer.
Let $M=2n+2$,
and suppose $\Re(s) > -n-\frac12+\frac{\varepsilon}{2}$,
with $\varepsilon>0$.
By \cite{Matsumoto}, we have
\begin{align}
\zeta_{SU(3)}^W(s) 
&= 
2^s \sum_{m,n\ge 1} \frac{1}{m^s n^s (m+n)^s}\\
&=
2^s \frac{\Gamma(2s-1) \Gamma(1-s)}{\Gamma(s)} \zeta(3s-1)
\nonumber \\
& + 2^s \sum_{k=0}^{M-1} (-1)^k \frac{s(s+1)\cdots(s+k-1)}{k!} \zeta(2s+k) \zeta(s-k) 
\nonumber\\
& + 2^s \frac{1}{2\pi \sqrt{-1}} \int_{\Re(z) = 2n+2-\varepsilon}
\frac{\Gamma(s+z) \Gamma(-z)}{\Gamma(s)} \zeta(2s+z) \zeta(s-z) dz.
\nonumber
\end{align}
Reminding 
\begin{equation}
\left. \frac{\Gamma(2s-1)}{\Gamma(s)} \right|_{s=-n}
= (-1)^{n-1} \frac{n!}{2 (2n+1)!},
\end{equation}
we can put $s=-n$ in this identity
and obtain
\begin{align}
\zeta_{SU(3)}^W(-n) &=
2^{-n} (-1)^{n-1} \frac{n! n!}{2 (2n+1)!} \zeta(-3n-1) \nonumber\\
&+ 2^{-n} \sum_{k=0}^{2n} (-1)^k \frac{(-n)(1-n)\cdots(k-1-n)}{k!} \zeta(-2n+k) \zeta(-n-k)\nonumber \\
&+ 2^{-n} (-1) \frac{(-n)(1-n) \cdots (-1) \cdot 1 \cdots n}{(2n+1)!} \frac12 \zeta(-3n-1).
\end{align}
This shows $\zeta_{SU(3)}^W(-n)=0$ for a positive odd integer $n$,
since $\zeta(-3n-1)=0$
and $\zeta(-2n+k) \zeta(-n-k)=0$ for $k=0,1,\dots,n$.
On the other hand,
for a positive even integer $n$,
we have
\begin{align}
\zeta_{SU(3)}^W(-n)
&= - 2^{-n} \frac{(n!)^2}{(2n+1)!} \zeta(-3n-1) \nonumber\\
& \qquad + 2^{-n} \sum_{k=0}^n \binom{n}{k} \zeta(-2n+k) \zeta(-n-k) =0,
\end{align}
where
the last equality follows from the following lemma:
\begin{lemma}\label{lemma}
For a positive even integer $n$,
we have
\begin{equation}
\sum_{k+l=n, k,l \ge 0} \frac{1}{k!l!} \zeta(-n-k) \zeta(-n-l)
= \frac{n!}{(2n+1)!} \zeta(-3n-1).
\end{equation}
\end{lemma}

Equivalently,
\begin{equation}
\sum_{k+l=n, k,l \ge 0} \frac{1}{k!l!} \frac{B_{n+1+k}}{n+1+k} \frac{B_{n+1+l}}{n+1+l} =
- \frac{n!}{(2n+1)!} \frac{B_{3n+2}}{3n+2}.
\end{equation}
This follows from \cite[Theorem~2]{CW}
when we substitute $\alpha=\gamma=n-1$ and $\delta=\varepsilon=1$.\qed

This concludes the proof of Theorem~\ref{Theorem:SU(3)}. 

\section{The groups over $\Z_p$}\label{Zp}

\subsection{$SL_2$}

Let $p$ be an odd prime.
We denote by $\Z_p$
the ring of integers in the non-archimedean local field $\Q_p$.
Jaikin-Zapirain \cite{Jaikin} obtains the following explicit formula:
\begin{equation}
\zeta_{SL_2(\Z_p)}^W(s) 
= Z_0(s) + Z_\infty(s),
\end{equation}
with
\begin{align}
Z_0(s) 
&= \zeta_{SL_2(\F_p)}^W(s) \nonumber\\
&= 1+2\left( \frac{p-1}{2} \right)^{-s} 
+ 2 \left( \frac{p+1}{2} \right)^{-s} 
+\frac{p-1}{2} \left( p-1 \right)^{-s} \nonumber\\
& \qquad +  p^{-s} + \frac{p-3}{2} (p+1)^{-s}, 
\\
Z_\infty(s)
&= \frac{1}{1-p^{-s+1}}\left(
4p \left( \frac{p^2-1}{2} \right)^{-s} 
+\frac{p^2-1}{2} (p^2-p)^{-s} \right. \nonumber\\
& \qquad \left. + \frac{(p-1)^2}{2} (p^2+p)^{-s}
\right).
\end{align}
This deduces
\begin{align}
Z_0(-2) &= p(p^2-1) = \left| SL_2(\F_p) \right|
=p(p+1)(p-1), \\
Z_\infty(-2) &= -p(p^2-1), \\
Z_0(-1) &= p(p+1), \\
Z_\infty(-1) &= -p(p+1), \\
Z_0(0) &= p+4, \\
Z_\infty(0) &= -\frac{4}{p-1} -p -4.
\end{align}
This shows
\begin{align}
\zeta_{SL_2(\Z_p)}^W(-2) &=0, \\
\zeta_{SL_2(\Z_p)}^W(-1) &=0, \\
\zeta_{SL_2(\Z_p)}^W(0) &=-\frac{4}{p-1},
\end{align}
which concludes the proof of Theorem~\ref{Theorem:SL(Zp)}.

\subsection{Congruence subgroups of $SL_2$}
In this subsection, we assume that $p$ is an odd prime.
By \cite{AKOV},
we obtain
\begin{align}
\zeta_{SL_2(\Z_p)[p^m]}^W(s)
&= p^{3m+2} \frac{1-p^{-2-s}}{1-p^{1-s}}.
\end{align}
This shows
\begin{align}
\zeta_{SL_2(\Z_p)[p^m]}^W(-2) &= 0, \\
\zeta_{SL_2(\Z_p)[p^m]}^W(-1) &= -p^{3m+1}/(p+1).
\end{align}
By taking an ``absolute limit'' $p\to 1$, we obtain
\begin{align}
\zeta_{SL_2(\Z_1)[1^m]}^W(s) &= \frac{s+2}{s-1}.
\end{align}

\subsection{Congruence subgroups of $SL_3$ and $SU_3$}

In this subsection, we assume that $p$ is a prime with $p\neq 3$.
By \cite{AKOV}, we have 
\begin{align}
\zeta_{SL_3(\Z_p)[p^m]}^W(s)
&= p^{8m} \frac{1+u(p) p^{-3-2s} + u(p^{-1}) p^{-2-3s} + p^{-5-5s}}{(1-p^{1-2s})(1-p^{2-3s})},
\end{align}
where $u(X) = X^3+X^2-X-1-X^{-1}$.
We notice that it can be factorized as
\begin{align}
\zeta_{SL_3(\Z_p)[p^m]}^W(s)
&= p^{8m} \frac{(1-p^{-2-s})(1-p^{-1-s})}{(1-p^{1-2s})(1-p^{2-3s})} \nonumber\\
& \times \left(
1+(p^{-1}+p^{-2}) p^{-s} + (1+p^{-1}) p^{-2s}+p^{-2-3s}
\right).
\label{eq:SL3}
\end{align}
We see that
\[
\zeta_{SL_3(\Z_p)[p^m]}^W(-2)=
\zeta_{SL_3(\Z_p)[p^m]}^W(-1)=0.
\]
The formula (\ref{eq:SL3}) shows
\begin{equation}
\lim_{p \to 1} \zeta_{SL_3(\Z_p)[p^m]}^W(s)
=\frac{(s+1)(s+2)}{(s-\frac12)(s-\frac23)},
\end{equation}
which is considered to be
``an absolute Witten zeta function
$\zeta_{SL_3(\Z_1)[1^m]}^W(s)$''.

Also by \cite{AKOV}, 
\begin{align}
\zeta_{SU_3(\Z_p)[p^m]}^W(s)
&= p^{8m} \frac{1+u(p) p^{-3-2s} + u(p^{-1}) p^{-2-3s} + p^{-5-5s}}{(1-p^{1-2s})(1-p^{2-3s})} \\
&= p^{8m} \frac{(1-p^{-2-s})(1-p^{-s})(1+p^{-1-s})}{(1-p^{1-2s})(1-p^{2-3s})}\nonumber\\
& \quad\times \left(
1+(1-p^{-1}+p^{-2}) p^{-s} + p^{-2-2s}
\right),
\label{eq:SU3}
\end{align}
where $u(X) = -X^3+X^2-X+1-X^{-1}$.
This shows
\[
\zeta_{SU_3(\Z_p)[p^m]}^W(-2)=
\zeta_{SU_3(\Z_p)[p^m]}^W(0)=0,
\]
while
\begin{equation}
\zeta_{SU_3(\Z_p)[p^m]}^W(-1)
= 2p^{8m-2} \frac{p-1}{p^5-1}
= 2p^{8m-2} \frac{1}{[5]_p}
\end{equation}
is non-zero
where $[n]_p = \frac{p^n-1}{p-1}$ is a $p$-analogue of an integer $n$.
This shows
\begin{equation}
\lim_{p\to 1} \zeta_{SU_3(\Z_p)[p^m]}^W(-1) = \frac25.
\end{equation}
By the formula (\ref{eq:SU3}), we have
\[
\lim_{p\to 1} \zeta_{SU_3(\Z_p)[p^m]}^W(s) = \frac{s(s+2)}{(s-\frac12)(s-\frac23)}.
\]



\noindent
Nobushige KUROKAWA\\
Department of Mathematics, Tokyo Institute of Technology, \\
Oh-Okayama, Meguro, Tokyo, 152-8551, Japan.\\
{\tt kurokawa@math.titech.ac.jp}\\
\\
Hiroyuki OCHIAI\\
Faculty of Mathematics, Kyushu University, \\
Motooka, Fukuoka, 819-0395, Japan. \\
{\tt ochiai@imi.kyushu-u.ac.jp}

\end{document}